\documentclass{article}

\usepackage {amsmath}
\usepackage {hyperref}
\usepackage {amsfonts}
\usepackage {amsthm}
\usepackage[flushmargin]{footmisc}
\usepackage[parfill]{parskip}
\usepackage {amssymb}
\usepackage {fullpage}
\usepackage {mathrsfs}
\usepackage {amscd}
\usepackage{bbm}
\usepackage[all,cmtip]{xy}
\DeclareMathAlphabet{\mathpzc}{OT1}{pzc}{m}{it}

\title{Modular categories, crossed S-matrices and Shintani descent}
\author{Tanmay Deshpande}
\date{}

\newtheorem {thm} {Theorem} [section]
\newtheorem {prop} [thm] {Proposition}
\newtheorem {conj} [thm] {Conjecture}
\newtheorem {lem} [thm] {Lemma}
\newtheorem {cor} [thm] {Corollary}

\theoremstyle{definition}
\newtheorem {defn} [thm] {Definition}

\newtheorem {prob} [thm]  {Problem}

\newtheorem {rk} [thm]  {Remark}
\newtheorem {ex} [thm] {Example}

\newcommand{\beq}{\begin{equation}}
\newcommand{\eeq}{\end{equation}}
\newcommand{\bthm}{\begin {thm}}
\newcommand{\ethm}{\end {thm}}
\newcommand{\bprop}{\begin {prop}}
\newcommand{\eprop}{\end {prop}}
\newcommand{\bprob}{\begin {prob}}
\newcommand{\eprob}{\end {prob}}
\newcommand{\bcor}{\begin {cor}}
\newcommand{\ecor}{\end {cor}}
\newcommand{\blem}{\begin{lem}}
\newcommand{\elem}{\end{lem}}
\newcommand{\bdefn}{\begin{defn}}
\newcommand{\edefn}{\end{defn}}
\newcommand{\brk}{\begin{rk}}
\newcommand{\erk}{\end{rk}}

\newcommand{\ZbN}{\Z/2N\Z}
\newcommand{\ZN}{\Z/N\Z}

\newcommand{\xto}{\xrightarrow}
\newcommand{\onto}{\twoheadrightarrow}

\newcommand{\KabD}{K_{\Qab}(\D,F)}
\newcommand{\KabC}{K_{\Qab}(\C,F)}

\newcommand{\bpf}{\begin{proof}}
\newcommand{\epf}{\end{proof}}
\newcommand{\bex}{\begin{ex}}
\newcommand{\eex}{\end{ex}}
\newcommand{\rar}[1]{\stackrel{#1}{\longrightarrow}}
\newcommand{\f}{\mathbb}

\newcommand{\<}{\langle}
\renewcommand{\>}{\rangle}

\newcommand{\tensor}{\otimes}
\newcommand{\h}{\operatorname}
\newcommand{\e}{\operatorname{ev}}
\renewcommand{\c}{\operatorname{coev}}

\newcommand{\Sh}{\operatorname{Sh}}

\newcommand{\Fq} {\mathbb{F}_q}

\newcommand{\Fqcl} {\overline{\mathbb{F}}_q}
\newcommand{\Q} {\mathbb{Q}}

\newcommand{\F} {\mathbb{F}}
\newcommand{\un} {\mathbbm{1}}

\renewcommand{\phi} {\varphi}

\newcommand{\M} {\mathscr{M}}

\newcommand{\D} {\mathscr{D}}
\newcommand{\Deq} {\mathscr{D}^{\Z/2N\Z}}
\newcommand{\SD} {S\left({\D}^{\Z/2N\Z}\right)}
\newcommand{\C} {\mathscr{C}}

\newcommand{\Hom} {\h{Hom}}
\newcommand{\uuPic} {\underline{\underline{\h{Pic}}}}
\newcommand{\uPic} {\underline{\h{Pic}}}
\newcommand{\Pic} {{\h{Pic}}}
\newcommand{\uEqBr} {\underline{\h{EqBr}}}
\newcommand{\uEqMod} {\underline{\h{EqMod}}}
\newcommand{\EqMod} {{\h{EqMod}}}

\renewcommand{\Vec}{\h{Vec}}

\newcommand{\Qab} {\mathbb{Q}^{\h{ab}}}

\newcommand{\G}{{\Gamma}}

\newcommand{\noin}{\noindent}

\newcommand{\g} {{\gamma}}
\renewcommand{\l} {{\lambda}}

\newcommand{\ga} {{\gamma_1}}
\newcommand{\gb} {{\gamma_2}}

\newcommand{\id}{\operatorname{id}}

\newcommand{\bit}{\begin{itemize}}
\newcommand{\eit}{\end{itemize}}

\newcommand{\DG}{\D_G(G)}

\renewcommand{\O}{\mathcal{O}}

\newcommand{\Z}{\mathbb{Z}}
\newcommand{\bZ}{\mathbb{Z}}

\newcommand{\tr}{\h{tr}}
\newcommand{\Fun}{\h{Fun}}

\newcommand{\bconj}{\begin{conj}}
\newcommand{\econj}{\end{conj}}

\begin{document}
\maketitle

\begin{abstract}
\noin Let $\mathscr{C}$ be a modular tensor category over an algebraically closed field $k$ of characteristic 0. Then there is the ubiquitous notion of the S-matrix $S(\C)$ associated with the modular category. The matrix $S(\C)$ is a symmetric matrix, its entries are cyclotomic integers and the matrix $(\dim \C)^{-\frac{1}{2}}\cdot S(\C)$ is a unitary matrix. Here $\dim \mathscr{C}\in k$ denotes the categorical dimension of $\C$ and it is a totally positive cyclotomic integer. Now suppose that we also have a modular autoequivalence $F:\mathscr{C}\to \mathscr{C}$. In this paper, we will define and study the notion of a crossed S-matrix associated with the modular autoequivalence $F$. We will see that the crossed S-matrix occurs as a submatrix of the usual S-matrix of some ``bigger'' modular category and hence the entries of a crossed S-matrix are also cyclotomic integers. We will prove that the crossed S-matrix (normalized by the factor $(\dim \C)^{-\frac{1}{2}}$) associated with any modular autoequivalence is a unitary matrix. We will also prove that the crossed S-matrix is essentially the ``character table'' of a certain semisimple commutative Frobenius $k$-algebra associated with the modular autoequivalence $F$. The motivation for most of our results comes from the theory of character sheaves on algebraic groups, where we expect that the transition matrices between irreducible characters  and character sheaves can be obtained as certain crossed S-matrices. In the character theory of algebraic groups defined over finite fields, there is the notion of Shintani descent of Frobenius stable characters. We will define and study a categorical analogue of this notion of Shintani descent in the setting of modular categories.
\end{abstract}

%\tableofcontents

\section{Introduction}
Let $k$ be an algebraically closed field of characteristic 0. Let $\C$ be a $k$-linear modular category. Let $F:\mathscr{C}\to \mathscr{C}$ be a modular autoequivalence.  We will associate with this modular autoequivalence a square\footnote{The rows and columns of this matrix are indexed by two rather different sets. We will see later that these two sets have the same cardinality, although this is not at all clear a priori.} matrix $S(\C,F)$ (well defined up to scaling rows by roots of unity) with entries in $k$ known as the (unnormalized) crossed S-matrix. Let $\O_\C$ denote the set of (isomorphism classes of) simple objects in $\C$. Then $F$ induces a permutation (also denoted by $F$) of the set $\O_\C$. The size of the square matrix $S({\C,F})$ is equal to $|\O_\C^F|$, the number of $F$-stable simple objects in $\C$. We remark that the matrix $S(\C,\id_\C)$ is the usual S-matrix (cf.  \cite{BK}, \cite{ENO1}) of the modular category $\C$. We will prove that the matrix $S(\C,F)$ occurs as a submatrix in the S-matrix of a certain ``bigger'' modular category. Using this we will deduce that the entries of the crossed S-matrix must in fact be cyclotomic integers. In particular we can unambiguously talk about the conjugate transpose $\overline{S(\C,F)}^T$ of the matrix $S(\C,F)$. We will prove that
\beq
S(\C,F)\cdot \overline{S(\C,F)}^T=\dim \C\cdot I=\overline{S(\C,F)}^T\cdot S(\C,F),
\eeq
i.e. the normalized crossed S-matrix $(\dim \C)^{-\frac{1}{2}}\cdot S(\C,F)$ is unitary.

Associated with the modular autoequivalence $F$, we will also define a commutative Frobenius $k$-algebra $K(\C,F)$ using a  Grothendieck ring construction. This algebra has a basis parametrized by the set $\O_\C^F$. We will prove that the algebra $K(\C,F)$ is semisimple and that the ``character table'' of this algebra is given by the crossed S-matrix $S(\C,F)$. Using this we will prove an analogue of the Verlinde formula for the algebra $K(\C,F)$. 

One of the motivations for studying these structures comes from the theory of character sheaves on algebraic groups, where we expect that the transition matrices between irreducible characters  and character sheaves can be obtained as certain crossed S-matrices. Throughout this paper we will provide references to analogous results and definitions from the theory of character sheaves which serve as a motivation for the results and definitions of the present paper.

We now briefly describe the organization of this paper. In \S\ref{s:mftcs} we describe the motivation for our results from the theory of character sheaves and some history. In \S\ref{s:pamc} we briefly recall the notion of a modular category and related ideas. In \S\ref{s:maimc} we briefly describe some results from \cite{ENO2} which relate the notion of modular autoequivalences of $\C$ to the notion of invertible  $\C$-module categories with traces. In \S\ref{s:dmr} we describe our main definitions and results. In \S\ref{s:crsmat} we define the crossed S-matrix $S(\C,F)$ using the invertible $\C$-module category $\M$ that we can associated with $F$ in view of \S\ref{s:maimc}. We also show that we can lift the modular autoequivalence $F$ to define a modular action of a cyclic group $\Z/N\Z$ on $\C.$ In \S\ref{s:faawma}, \ref{s:sfa} we introduce the commutative Frobenius $k$-algebra $K(\C,F)$ associated with the modular autoequivalence $F$, study some of its preliminary properties and prove that it is semisimple. In \S\ref{s:csbcc} we will construct an auxiliary spherical braided $\Z/2N\Z$-crossed category $\D$ with trivial component $\C$. The equivariantization $\D^{\ZbN}$ of $\D$ turns out to be a modular category which will be used in some of our definitions and proofs. In \S\ref{s:ucsm} we state our main results about the matrix $S(\C,F)$ and the commutative algebra $K(\C,F)$ and its characters. In \S\ref{s:pmrg} we will prove the main results. In fact we will state and prove more general versions of our main results. In \S\ref{s:fak}, \ref{s:cak} we introduce a certain Frobenius $k$-algebra $K(\D,F)$ such that $K(\C,F)\subset K(\D,F)$ is a subalgebra and study the characters of the algebra $K(\D,F)$. In \S\ref{s:msmf}, \ref{s:umsmf} we introduce some more general crossed S-matrices and prove that in fact all of them are unitary (up to normalization). In \S\ref{s:cp} we complete the proofs of our main results. In \S\ref{s:sdmc} we introduce the notion of Shintani descent in our setting of modular categories equipped with an autoequivalence. This is motivated by the notion of Shintani descent from the character theory of algebraic groups defined over finite fields (cf. \cite{S},\cite{De2}). To define Shintani descent, in \S\ref{s:tsbcc} we introduce the notion of ribbon twists in braided crossed categories. In \S\ref{s:sm}, \ref{s:rwmsmf} for each positive integer $m$ we define the $m$-th Shintani matrix and  study its relationship with the generalized crossed S-matrices introduced in \S\ref{s:pmrg}. In \S\ref{s:sd}, \ref{s:tosd} we define the $m$-th Shintani descent map and the $m$-th Shintani basis of $K(\C,F)$ for each positive integer $m$. Finally, we prove that there exists a positive integer $m_0$ such that the $m$-th Shintani basis of $K(\C,F)$ essentially only depends on the residue of $m$ modulo $m_0$ and also that the 1st Shintani descent map is related to the Asai twisting operator (cf. \cite{S}, \cite{De2}), just as in the character theory of algebraic groups over finite fields.

\section*{Acknowledgments}
I am grateful to V. Drinfeld, M. Boyarchenko, and V. Ostrik for useful correspondence. This work was supported by World Premier Institute Research Center Initiative (WPI), MEXT, Japan.

\subsection{Motivation from the theory of character sheaves}\label{s:mftcs}
In order to study the relationship between character sheaves and characters of finite reductive groups, Lusztig introduced a ``nonabelian Fourier transform'' matrix for any finite group $\Gamma$ (see \cite{L1}). This is a square matrix whose rows as well as columns are parametrized by the isomorphism classes of simple $\Gamma$-equivariant (under conjugation) vector bundles on $\Gamma$. On the other hand, the category $\Sh_\G(\G)$ of $\G$-equivariant vector bundles on $\G$ has the structure of a modular category under convolution of vector bundles. The nonabelian Fourier transform matrix can then be interpreted as the S-matrix of this modular category.

Now suppose that we have an automorphism $F:\G\rar{}\G$ of finite groups. Lusztig also defined a nonabelian Fourier transform matrix associated with such a group automorphism. Note that any such automorphism induces an autoequivalence $F:={F^{-1}}^*:\Sh_\G(\G)\rar{}\Sh_\G(\G)$ of modular categories. We will see that we can interpret the nonabelian Fourier transform matrix in this case as the crossed S-matrix associated with the modular autoequivalence $F$.

Lusztig associated a finite group $\Gamma$ with every two-sided cell in a Weyl group. Thus he associated a nonabelian Fourier transform matrix with each two-sided cell in any Weyl group. Later in \cite{L3}, \cite{GM} the Fourier matrices were generalized to the case of any finite Coxeter system equipped with an automorphism (also sometimes called twist). Thus given any stable two-sided cell in any finite Coxeter system with automorphism, we have an associated Fourier transform matrix. Lusztig used the Fourier transform matrices associated with Weyl groups to transition between unipotent irreducible characters and unipotent almost characters of the corresponding finite reductive groups.  

We expect that all the Fourier matrices associated with stable two-sided cells in finite twisted Coxeter groups can be realized as certain crossed S-matrices. To be more precise, we expect that given a stable two-sided cell in a finite twisted Coxeter group, we should be able to construct a modular category equipped with a modular autoequivalence, for example using the theory of Soergel bimodules (cf. \cite[\S7]{L4}). Then we expect to recover the Fourier matrix associated with this stable two-sided cell as the crossed S-matrix associated with the corresponding modular autoequivalence.

Let us now consider the case of unipotent groups defined over finite fields. In \cite{BD}, Boyarchenko and Drinfeld have developed a theory of character sheaves on unipotent groups. Let $G$ be a connected unipotent group over $\Fqcl$ and let $F:G\to G$ be an $\Fq$-Frobenius map. Then the irreducible characters of $G(\Fq)$ can be partitioned into families known as $\f{L}$-packets of irreducible characters. The $\f{L}$-packets are parametrized by the $F$-stable minimal idempotents in the braided triangulated category $\DG$. On the other hand, associated with each minimal idempotent $e\in \DG$ is a modular category $\M_e\subset e\DG$. The finite set of (isomorphism classes of) simple objects of $\M_e$ is known as the $\f{L}$-packet of character sheaves associated with $e$ and is denoted as $CS_e(G)$. If $e$ is an $F$-stable minimal idempotent then we have the modular autoequivalence 
\beq
F:={F^{*}}^{-1}:\M_e\rar{}\M_e.
\eeq
In \cite{De1} it was shown that the transition matrix between the irreducible characters in the $\f{L}$-packet associated with $e$ and the ``trace of Frobenius'' functions of $F$-stable character sheaves $CS_e(G)^F$ is the crossed S-matrix $S(\M_e,F)$. We refer to \cite{BD}, \cite{De1}, \cite{De2} for a detailed exposition.

In fact we expect that there should be an interesting theory of character sheaves on general algebraic groups over finite fields and that in general the transition matrices between characters and character sheaves should again be given by certain crossed S-matrices.

\subsection{Preliminaries on modular categories}\label{s:pamc}
In this section we only briefly recall the notion of a modular category. We refer to \cite{BK}, \cite{ENO1}, \cite{DGNO} for a detailed exposition of the theory of fusion categories, braided fusion categories and modular categories.  

Let $\C$ be a semisimple abelian $k$-linear category with finitely many simple objects. Let $\otimes:\C\times \C\rar{}\C$ be an associative tensor product with unit object $\un$ that gives $\C$ the structure of a monoidal category. Let us assume that $\un$ is a simple object and that $\C$ is a rigid monoidal category. This means that every object $C\in \C$ has a rigid dual, namely an object $C^*\in \C$, an evaluation map $\e_C:C^*\otimes C\rar{}\un$ and a coevaluation map $\c_C:\un\rar{}C\otimes C^*$  satisfying some compatibility relations.

A category $\C$ as above is known as a fusion category. The notion of the categorical dimension $\dim\C\in k$ of a fusion category $\C$ is defined in \cite{ENO1}. By the results of \cite{ENO1}, the categorical dimension $\dim \C\in k$ is in fact a totally positive cyclotomic integer. 

In any fusion category the double duality functor $C\mapsto C^{**}$ is a monoidal autoequivalence of $\C$. A pivotal structure on $\C$ is a monoidal isomorphism 
\beq
\psi:\id_\C\rar{\cong}(\cdot)^{**}.
\eeq
In any pivotal category we can define categorical dimensions $\dim_\pm(C)$ of objects $C\in \C$. Let $\Q^{\h{ab}}\subset k$ denote the subfield of cyclotomic numbers, i.e. the subfield generated by all the roots of unity in $k$. By the Kronecker-Weber theorem $\Qab$ is the abelian closure of $\Q$. The subfield $\Qab\subset k$ has a special automorphism known as ``complex conjugation''  which maps each root of unity to its inverse.

By \cite[\S2]{ENO1}, $\dim_\pm(C)\in \Qab$ and we have $\overline{\dim_+(C)}=\dim_+(C^*)=\dim_-(C)$. We say that the pivotal structure is spherical if $\dim_+(C)=\dim_+(C^*)$. In this paper we will mostly only encounter spherical fusion categories and in this case we will simply denote the dimension as $\dim$.  In any spherical fusion category, the categorical dimensions of all objects in $\C$ are totally real cyclotomic integers.

We say that a fusion category is braided if it is also equipped with braiding isomorphisms $$\beta_{C_1,C_2}:C_1\otimes C_2\rar{\cong}C_2\otimes C_1$$ functorial in $C_1,C_2\in \C$ and satisfying some compatibility relations. A pre-modular category $\C$ is a braided fusion category equipped with a spherical structure. The notion of a spherical structure on a braided fusion category is equivalent to that of a ribbon twist structure. A ribbon twist in a braided fusion category $\C$ is a natural isomorphism $\theta:\id_\C\to \id_C$ such that $\theta_{C^*}=\theta^*_C$ for all $C\in \C$ and
\beq
\theta_{C_1\otimes C_2}=(\theta_{C_1}\otimes \theta_{C_2})\circ \beta_{C_2,C_1}\circ \beta_{C_1,C_2}\mbox{ for any $C_1,C_2\in \C$}.
\eeq

The (unnormalized) S-matrix $S(\C)$ of a pre-modular category $\C$ is an $\O_\C\times \O_\C$ matrix defined by $S(\C)_{C_1,C_2}=\tr(\beta_{C_2,C_1}\circ \beta_{C_1,C_2})$ for $C_1,C_2\in \O_\C$.  We say that a pre-modular category $\C$ is modular if its S-matrix is invertible.

\subsection{Modular autoequivalences and invertible module categories}\label{s:maimc}
Let $\C$ be a $k$-linear modular category with twist denoted by $\theta$. We will now recall some facts from \cite{ENO2} that we will use throughout this paper. We refer to {\it op. cit.} for a detailed account. Let $\uEqMod(\C)$ denote the categorical 1-group (i.e. a monoidal groupoid with invertible objects) whose objects are modular autoequivalences of $\C$ and morphisms are natural isomorphisms between modular autoequivalences. Let $\EqMod(\C)$ denote the group of modular autoequivalences of $\C$ up to isomorphism. We note that $\uEqMod(\C)$ is a full 1-subgroup of the categorical 1-group $\uEqBr(\C)$ of all braided autoequivalences of $\C$ that is studied in \cite{ENO2}. 

Following \cite{ENO2}, let $\uuPic(\C)$ denote the categorical 2-group whose objects are invertible $\C$-module categories, 1-morphisms are equivalences of $\C$-module categories and 2-morphisms are natural isomorphisms. Let $\uuPic^{\tr}(\C)\subset \uuPic(\C)$ be the full 2-subgroup formed by those invertible $\C$-module categories which can be equipped with a $\C$-module trace. A $\C$-module trace $\tr_\M$ on $\M$ assigns a trace $\tr_M(f)\in k$ for each endomorphism $f:M\to M$ in $\M$ satisfying some properties (cf. \cite{Sch}) and in particular we can talk of dimensions $\dim_\M$ of objects of $\M$. The compatibility of the trace $\tr_\M$ with the spherical structure on $\C$ means that we must have $\dim_\M(C\otimes M)=\dim(C)\cdot \dim_\M(M)$ whenever $C\in \C$ and $M\in \M$. If such a $\C$-module trace exists, then it must be unique up to scaling. We will often assume that the trace is normalized in such a way that
\beq\label{e:trnorm}
\dim(\underline{\Hom}(M,N))=\dim_\M(M)\cdot \dim_\M(N) \mbox{ for $M,N\in \M$,}
\eeq
where $\underline\Hom(M,N)\in \C$ is the internal Hom. With this additional condition, $\tr_\M$ is uniquely defined up to scaling by $\pm 1$ and we have $\sum\limits_{M\in \O_\M}\dim_\M(M)^2=\dim \C=\sum\limits_{C\in \O_\C}\dim(C)^2$, where $\O_\M$ denotes the set of isomorphism classes of simple objects of $\M$. Moreover with such a normalization, $\dim_\M(M)$ must be a totally real cyclotomic integer for each $M\in \M$.

Now by \cite[Thm. 5.2]{ENO2} the 1-truncation $\uPic(\C)$ is equivalent to $\uEqBr(\C)$ as a categorical 1-group. This induces an equivalence
\beq\label{e:eqmod}
\uEqMod(\C)\rar{\cong} \uPic^{\tr}(\C) \mbox{ of 1-subgroups.}
\eeq

Hence given a modular autoequivalence $F:\C\rar{}\C$, we have an associated invertible $\C$-module category $\M$ which admits a $\C$-module trace. 

We will define the crossed S-matrix associated with $F$ in terms of the $\C$-module category $\M$. In fact the rows of the crossed S-matrix $S(\C,F)$ are parametrized by the set $\O_\C^F$ whereas the columns are parametrized by $\O_\M$. We will prove that the matrix $(\dim \C)^{-\frac{1}{2}}\cdot S(\C,F)$  is unitary. As a corollary we obtain another proof of the fact that $|\O_\C^F|=|\O_\M|$. In view of (\ref{e:eqmod}), we may sometimes denote the crossed S-matrix by $S(\C,\M)$ instead of $S(\C,F)$.

\section{Definitions and main results}\label{s:dmr}
\subsection{The crossed S-matrix associated with a modular autoequivalence}\label{s:crsmat}
Let $F:\C\rar{}\C$ be a modular autoequivalence. In other words, we have an action of $\Z$ on $\C$ by modular autoequivalences. Let $\M$ be the invertible $\C$-module category associated with $F$ under (\ref{e:eqmod}). By definition this means that for $C\in \C$, $M\in \M$ we have functorial crossed braiding isomorphisms
\beq
\beta_{M,C}:M\otimes C \rar{\cong} F(C)\otimes M ,
\eeq
\beq
\beta_{C,M}:C\otimes M \rar{\cong} M\otimes C
\eeq
in the module category $\M$.

By \cite[Thm. 4.15]{ENO2}, there must be an $n$ such that $F^n\cong \id_\C$ giving us a group homomorphism 
\beq
\Z/n\Z\rar{}\EqMod(\C)\cong\Pic^{\tr}(\C).
\eeq 
However, by \cite[\S8]{ENO2} this may not lift to a modular action of $\Z/n\Z$ on $\C$, with the obstruction being given by some element of $H^3(\Z/n\Z,\O_\C^\times)$, where $\O_\C^\times$ is the group of isomorphism classes of invertible (and hence simple) objects of $\C$. We can choose a positive integer $N\in n\Z$ such that the pullback of the obstruction vanishes in  $H^3(\Z/N\Z,\O_\C^\times)$. We now choose a lift 
\beq
\Z/N\Z\rar{}\uEqMod(\C)\cong\uPic^{\tr}(\C).
\eeq
In particular we have a natural isomorphism $\id_\C\rar{\xi}F^N$ of modular autoequivalences and this induces an action of $\Z/N\Z$ on $\C$ by modular autoequivalences.

We will now define the crossed S-matrix $S(\C,\M)$. First we choose a $\C$-module trace $\tr_\M$ on $\M$ normalized according to (\ref{e:trnorm}). For each object $C\in \O_\C^F$, we fix an isomorphism $\psi_C:F(C)\rar{\cong}C$ such that $(C,\psi_C)$ determines an object of the $\Z/N\Z$-equivariantization $\C^{\Z/N\Z}$. In other words, we choose $\psi_C$ in such a way that the composition
\beq\label{e:psicondi}
C\rar{\xi_C} F^N(C)\xto{F^{N-1}(\psi_C)} F^{N-1}(C)\cdots\xto{F(\psi_C)}F(C)\xto{\psi_C}C
\eeq
equals $\id_C$. Hence we see that there are $N$ choices for $\psi_C$ which differ by scaling by $N$-th roots of unity. Furthermore, we will always make the convention that $\psi_\un=\id_\un$, although this is not strictly necessary.
\bdefn\label{d:crsmat}
For $C\in \O_\C^F$ and $M\in \O_\M$ we define the automorphism $\g_{C,M}$ as the composition
\beq
\g_{C,M}: C\otimes M\xto{\beta_{C,M}}M\otimes C\xto{\beta_{M,C}}F(C)\otimes M\xto{\psi_C\otimes \id_M} C\otimes M.
\eeq
The crossed S-matrix $S(\C,\M)$ is the $\O_\C^F\times \O_\M$ matrix whose $(C,M)$-th entry is 
\beq
S(\C,\M)_{C,M}:=\tr_\M(\g_{C,M})\in k.
\eeq
\edefn

\brk\label{r:de12diff}
(i) The crossed S-matrix depends on the choice of the $\psi_C$. A different choice of $\psi_C$ amounts to multiplying the row corresponding to $C$ by an $N$-th root of unity. Note that with our fixed choice of $\psi_\un=\id_\un$, we have $S(\C,M)_{\un,M}=\dim_\M (M)$.\\
(ii) The definition of the crossed S-matrix presented here is slightly different from that given in \cite{De1}. Firstly, in {\it op. cit.} we did not impose the restriction (\ref{e:psicondi}) on the isomorphisms $\psi_C$. Secondly, we had used the inverse of the crossed braidings instead of the crossed braidings in the definition of $\g_{C,M}$. This essentially amounts to replacing $C$ by its dual $C^*$, or as we will see, it essentially amounts to looking at the complex conjugate of the crossed S-matrix defined here. As of now, we only know that the entries of the crossed S-matrix lie in $k$. However, we will prove that in fact they lie in $\Qab$ and hence the complex conjugate of the crossed S-matrix is unambiguously defined.
\erk

\subsection{The Frobenius algebra associated with a modular autoequivalence}\label{s:faawma}
As seen before, the modular autoequivalence $F:\C\rar{}\C$ defines a modular action of $\Z$ on $\C$. Let us form the $\f{Z}$-equivariantization $\C^\Z$ from this action. The objects of $\C^\bZ$ can be thought of as pairs $(C,\psi)$, where $C\in \C$ and $\psi:F(C)\xto{\cong} C$. This is a (nonsemisimple) spherical braided monoidal abelian category. Hence the Grothendieck ring $K_0(\C^{\f{Z}})$ is a commutative ring. 

Let us consider the trivial modular category $\Vec$ of finite dimensional $k$-vector spaces equipped with the trivial action of $\f{Z}$. Then $\Vec^{\f{Z}}$ is the category of finite dimensional $k$-vector spaces $V$ equipped with an automorphism $\psi$, or equivalently an action of $\f{Z}$. We have the ring homomorphism $K_0(\Vec^\Z)\rar{}k$ that takes the class of $(V,\psi)\in \Vec^\Z$ to $tr(\psi)\in k$. We can identify $K_0(\Vec^\Z)$ with the group ring $\Z[k^\times]$ so that the element $[\alpha]\in\Z[k^\times]$ corresponding to group element $\alpha\in k^\times$ gets identified with the class of a $1$-dimensional vector space equipped with the automorphism of multiplication by $\alpha$. Our morphism $K_0(\Vec^\Z)\rar{}k$ takes $[\alpha]\in \Z[k^\times]$ to $\alpha\in k$.

For any modular category $\C$ equipped with a $\f{Z}$ action, we have the braided functor $\Vec^{\f{Z}}\to \C^{\f{Z}}$ which induces a ring homomorphism $K_0(\Vec^\Z)\rar{}K_0(\C^\Z)$. Now define the commutative $k$-algebra 
\beq
K({\C,F}):=K_0(\C^\Z)\otimes_{K_0(\Vec^\Z)}k.
\eeq 

\bdefn\label{d:linfun}
The algebra $K({\C,F})$ is equipped with a canonical linear functional $\lambda:K({\C,F})\rar{}k$ such that $\lambda(1)=1$ defined as follows:\\
The functor $\C^\bZ\rar{}\Vec^\bZ$ that takes $C\in\C^\bZ$ to $\Hom(\mathbf{1},C)\in\Vec^\bZ$ induces a morphism of $K_0(\Vec^\bZ)$-modules $K_0(\C^\bZ)\rar{}K_0(\Vec^\bZ)$ and therefore a $k$-linear functional $\lambda:K({\C,F})\rar{}k$ as desired.
\edefn

The following result was proved in \cite[\S7.1]{De1}. We recall it below along with the the proof for completeness.
\begin{prop}\label{p:de1frobalg} (\cite[Prop. 7.2]{De1})
(i) We have $\dim K({\C,F})=|\O_\C^F|$, where $\O_\C^F$ is the set of $F$-fixed (equivalently $\Z$-fixed) elements of ${\O_\C}$.\\
(ii) The bilinear form $(a_1,a_2)\longmapsto\lambda(a_1a_2)$ for $a_1,a_2\in K({\C,F})$, is nondegenerate. In other words, $K({\C,F})$ is a Frobenius algebra. 
\end{prop}

\begin{proof}
Let us describe $K_0(\C^{\bZ})$ as a module over $K_0(\Vec^\bZ)=\bZ[k^\times]$ in terms of the $\Z$-orbits $O$ in the set $\O_\C$ of simple objects of $\C$. We have
\[
\C=\bigoplus_O \C_O, \quad \C^\bZ=\bigoplus_O \C_O^\bZ,
\]
where $O$ runs through the $\bZ$-orbits in ${\O_\C}$ and $\C_O\subset\C$ is the full subcategory consisting of objects of $\C$, all of whose irreducible components are in $O$. Hence
\beq
K_0(\C^\bZ)=\bigoplus_O K_0(\C_O^\bZ).
\eeq
 We want to show that once we tensor with $k$, only the summands supported on the singleton orbits (i.e. $\O_\C^F$) survive. It is easy to see that if $|O|=n$, then $K_0(\C_O^\bZ)\subset K_0(\C^\bZ)$ is a $\bZ[k^\times]=K_0(\Vec^\Z)$-submodule, which is (noncanonically) isomorphic to $\bZ[k^\times]/I_n$, where $I_n\subset\bZ[k^\times]$ is the ideal generated by elements of the form $[\zeta]-1$, where $\zeta\in k^\times$ is such that $\zeta^n=1$. If $n>1$, then $\zeta^n=1$ for some $\zeta\in k\setminus\{1\}$, so $K_0(\C_O^\bZ)\tensor_{\bZ[k^\times]}k=0$. If $n=1$ i.e. if $O=\{C\}$ for some $C\in \O_\C^F$, then $\dim_{k} \left( K_0(\C_{\{C\}}^\bZ)\tensor_{\bZ[k^\times]}k\right)=1$. So $K({\C,F})$ is the direct sum of the lines $L_C:=K_0(\C^\bZ_{\{C\}})\tensor_{\bZ[k^\times]}k$, where $C$ runs over $\O_\C^F$. This implies statement (i).

Given $C,C'\in \O_{\C}^F$, the pairing $L_C\times L_{C'}\rar{}k$, induced by the bilinear form $(a_1,a_2)\longmapsto\lambda(a_1a_2)$, is nonzero if and only if $C'$ is dual to $C$. This implies (ii).
\end{proof}

\brk\label{r:connuni}
Let $G$ be a connected unipotent group over $\Fqcl$ equipped with a $\Fq$-Frobenius map $F$. In our analogy with the theory of character sheaves, the commutative Frobenius algebra $K(\C,F)$ is the analogue of the algebra of class functions on the finite group $G(\Fq)$. The map which sends an object $(C,\psi)\in \C^\Z$ to its class in the algebra $K(\C,F)$ is analogous to the sheaf-function correspondence. We refer to \cite[\S7]{De1} for details.
\erk

\subsection{Semisimplicity of the Frobenius algebra $K(\C,F)$}\label{s:sfa}
In this section we will prove that the Frobenius $k$-algebra $K(\C,F)$ is semisimple. In the previous section we defined the algebra $K(\C,F)$ by looking at the non-semisimple category $\C^\Z$. Let us now give an alternative description within the realm of fusion categories.

Recall from \S\ref{s:crsmat} that the modular autoequivalence $F$ can be lifted to define an action of $\Z/N\Z$ on $\C$ by modular autoequivalences. Now since $\Z/N\Z$ is a finite group, the equivariantization $\C^{\Z/N\Z}$ is a (possibly degenerate) braided fusion category.  Let us form the Grothendieck ring $K_0(\C^{\Z/N\Z})$. It is a based ring in the sense of \cite{L2} and hence $K_0(\C^{\Z/N\Z})\otimes \Q$ is a semisimple $\Q$-algebra (by \cite[\S1.2]{L2}). Let $\omega$ be a primitive $N$-th root of unity and let us extend scalars and consider the $\Z[\omega]$-algebra $K_0(\C^{\Z/N\Z})\otimes_\Z\Z[\omega]$. Now the pair $(\un, \omega:\un\rar{\cong}\un)$ is a simple object of $\C^{\Z/N\Z}$. Note that the class of $(\un,\id_\un)$ is the unit in $K_0(\C^{\Z/N\Z})\otimes_\Z\Z[\omega]$. Let us form the quotient algebras
\beq
K_{\Z[\omega]}(\C,F):=\frac{K_0(\C^{\Z/N\Z})\otimes_\Z\Z[\omega]}{\<[(\un,\omega)]-\omega\>} \mbox{ and } K_{\Q[\omega]}(\C,F):=\frac{K_0(\C^{\Z/N\Z})\otimes_\Z\Q[\omega]}{\<[(\un,\omega)]-\omega\>}.
\eeq
As in Definition \ref{d:crsmat}, for each $C\in \O_\C^F$ choose $\psi_C:F(C)\rar{\cong}C$ satisfying (\ref{e:psicondi}). Repeating the argument used in the proof of Proposition \ref{p:de1frobalg}, we see that the set of classes $\{[(C,\psi_C)]|C\in \O_\C^F\}$ forms a $\Z[\omega]$-basis of $K_{\Z[\omega]}(\C,F)$. As the quotient of a semisimple $\Q[\omega]$-algebra, $K_{\Q[\omega]}(\C,F)$ is also a semisimple $\Q[\omega]$-algebra. 

Using a similar argument to the one in \S\ref{s:faawma}, we have a ring homomorphism 
\beq
\Z\mu_N=K_0(\Vec^{\Z/N\Z})\rar{}K_0(\C^{\Z/N\Z}),
\eeq
where $\Z\mu_N$ denotes the group ring of the group $\mu_N$ of $N$-th roots of unity. We have the `tautological' map $\Z\mu_n\rar{}\Z[\omega]$. It is clear that we have an identification
\beq
K_{\Z[\omega]}(\C,F)\=K_0(\C^{\Z/N\Z})\otimes_{K_0(\Vec^{\Z/N\Z})}\Z[\omega].
\eeq

Combining all our observations, we deduce
\blem\label{l:altdes}
We have an isomorphism of $k$-algebras
\beq
K(\C,F)\cong K_{\Z[\omega]}(\C,F)\otimes_{\Z[\omega]}k.
\eeq
The algebra $K(\C,F)$ is semisimple. It has a basis $\{[(C,\psi_C)]\}_{C\in \O_\C^F}$ with respect to which the multiplication in $K(\C,F)$ has structure constants that lie in $\Z[\omega]$, i.e. for $C,C'\in \O_\C^F$
\beq\label{e:strucconst}
[(C,\psi_C)]\cdot [(C',\psi_{C'})]=\sum\limits_{D\in \O_\C^F} a_{C,C'}^D[(D,\psi_D)] \mbox{ with $a_{C,C'}^D\in \Z[\omega]$,}
\eeq
and the basis element $[(\un,\psi_\un=\id_\un)]$ is the unit of $K(\C,F)$.
\elem

\brk
The semisimplicity of the commutative $k$-algebra $K(\C,F)$ implies that it must be isomorphic to $k^{|\O_\C^F|}$ as an algebra. In other words, $K(\C,F)$ has a basis formed by the minimal idempotents $e\in K(\C,F)$.  Moreover, if $e,e'\in K(\C,F)$ are minimal idempotents, we have $\l(ee')=\delta_{e,e'}\l(e)$. Hence by Proposition \ref{p:de1frobalg}(ii) we must have $\l(e)\neq 0$ for each minimal idempotent $e\in K(\C,F)$. We will prove that the set of the minimal idempotents in $K(\C,F)$ can be parametrized by the set $\O_\M$ of simple objects of the $\C$-module category $\M$. We will see that the matrix relating the basis formed by the minimal idempotents (suitably normalized) in $K(\C,F)$ to the basis formed by  $\{[(C,\psi_C)]\}_{C\in \O_\C^F}$ is precisely the crossed S-matrix $S(\C,F)$. (See Corollary \ref{c:minid}.)
\erk

\subsection{Construction of a spherical braided crossed category}\label{s:csbcc}
In \S\ref{s:crsmat}, we defined a modular action of $\Z/N\Z$ on $\C$:
\beq\label{e:modaction}
\Z/N\Z\rar{}\uEqMod(\C)\cong\uPic^{\tr}(\C).
\eeq

Now note that the group cohomology $H^4(\Z/N\Z,k^\times)$ of the cyclic group $\Z/N\Z$ is trivial. Hence by \cite[\S8]{ENO2} we can lift the modular action to a map of 2-groups
\beq\label{e:d'}
\Z/N\Z\rar{}\uuPic^{\tr}(\C).
\eeq
We choose one such lift. Let $\D'$ denote the corresponding braided $\Z/N\Z$-crossed category (cf. \cite{ENO2}) so that we have
\beq
\D'=\bigoplus\limits_{a\in \Z/N\Z}\M_a
\eeq
with $\C=\M_0$ and $\M=\M_1$.

We refer to \cite{DGNO} for more on the notion of braided crossed categories. We recall that in particular this means that we have a monoidal action of $\Z/N\Z$ on $\D'$ agreeing with the braided monoidal action of $\ZN$ on $\C=\M_0$. For $a\in \ZN$, we let $F^a:\D'\rar{}\D'$ denote the action of $a$ on $\D'$. Since $\ZN$ is an abelian group, the monoidal action of $\ZN$ preserves each $\M_a$. If $L\in \M_a$, then $F^a(L)\cong L$. Moreover, for any $a\in \ZN$, $L\in\M_a$ and $D\in \D'$ we have functorial crossed braiding isomorphisms
\beq
\beta_{L,D}:L\otimes D\rar{\cong}F^a(D)\otimes L
\eeq
that satisfy certain compatibility relations.

In \S\ref{s:crsmat}, we have also chosen a normalized $\C$-module trace $\tr_\M$ on $\M$. For each $a\in \Z$, this induces a normalized $\C$-module trace $\tr_{\M^{\boxtimes a}}$ on the invertible $\C$-module category $\M^{\boxtimes a}$. Now (\ref{e:modaction}) gives us (an isomorphism class of) a $\C$-module equivalence $\C\cong \M^{\boxtimes N}$. Now the normalized $\C$-module trace $\tr_{\M^{\boxtimes N}}$ on $\M^{\boxtimes N}$ either agrees with the spherical trace in $\C$ or is the opposite of it. If it turns out to be the opposite, it means that $\D'$ cannot be equipped with a spherical structure that agrees with the one on $\C$ as well as with the chosen $\C$-module trace on $\M$.

\brk 
If the spherical structure on $\C$ is positive, then we can choose the positive normalized $\C$-module trace on $\M$ and in this case the trace $\tr_{\M^{\boxtimes N}}$ must necessarily agree with the (positive) spherical trace on $\C$. Consequently, if $\C$ is positive spherical then $\D'$ can be equipped with a spherical structure compatible with that on $\C$.
\erk

In any case, the trace $\tr_{\M^{\boxtimes 2N}}$ on $\M^{\boxtimes 2N}$ agrees with the spherical trace on $\C$.
Now consider the map of 2-groups  
\beq\label{e:2gpd}
\Z/2N\Z\onto\Z/N\Z\rar{}\uuPic^{\tr}(\C)
\eeq
induced from (\ref{e:d'}). Let $\D$ denote the corresponding braided $\Z/2N\Z$-crossed category so that we have
\beq
\D=\bigoplus\limits_{a\in \Z/2N\Z}\M_a
\eeq
with $\M_1=\M$. Moreover for each $a\in \Z/2N\Z$, $\M_a$ is equipped with a $\C$-module trace that is compatible with the tensor products in $\D$ and extending the traces on $\C,\M$. In other words the traces define a spherical structure on $\D$. From (\ref{e:2gpd}) we see that we have an identification $\M_a\cong \M_{a+N}$ as $\C$-module categories and that the action of $\Z/2N\Z$ on $\D$ in fact comes from an action of $\Z/N\Z$ on $\D$. Note however that the $\C$-module traces on $\M_a$ and $\M_{a+N}$ may differ by a sign. 

By \cite[Prop. 4.56]{DGNO}, the equivariantization $\D^{\Z/2N\Z}$ is a modular category. We will usually denote objects of $\Deq$ as pairs $(D,\psi)$ with $D\in \D$ and $\psi:F(D)\rar{\cong}D$. For $a\in\ZbN$, we denote by $\psi^{(a)}:F^a(D)\rar{\cong}D$, the isomorphism obtained from $\psi$. We will see that the (unnormalized) crossed S-matrix $S(\C,F)$ occurs as a submatrix in the (unnormalized) S-matrix of the modular category $\D^{\Z/2N\Z}$.

%In fact for each $a\in \Z/2N\Z$ we will now define more general crossed S-matrices $S(\M_a,F)$ which also occur as a submatrices in the S-matrix of $\D^{\Z/N\Z}$ and satisfy properties similar to that of $S(\C,F)$.

%\brk
%In the setting of character sheaves, these more general crossed S-matrices are very closely related to the notion of Shintani descent (cf. \cite{De2}).
%\erk

\subsection{Unitarity of the crossed S-matrix}\label{s:ucsm}
Let us now state our first main results. Recall that we have introduced the spherical braided $\Z/2N\Z$-crossed category $\D$ and the modular category $\D^{\Z/2N\Z}$. Let $S({\D^{\Z/2N\Z}})$ denote the S-matrix of this modular category. Note also that for each $C\in \O_\C^F$, the pair $(C,\psi_C)$ defines a simple object of $\C^{\Z/N\Z}\subset \C^{\Z/2N\Z}\subset \D^{\Z/2N\Z}$. Similarly for each $M\in \O_\M$, let us choose an isomorphism $\psi_M:F(M)\rar{\cong}M$ such that we have
\beq\label{e:psimcondi}
\id_M:M\rar{\xi_M} F^N(M)\xto{F^{N-1}(\psi_M)} F^{N-1}(M)\cdots\xto{F(\psi_M)}F(M)\xto{\psi_M}M.
\eeq
Then the pair $(M,\psi_M)$ defines a simple object in $\M^{\Z/N\Z}\subset \D^{\Z/2N\Z}$.

\bthm\label{t:main}
(i) The crossed S-matrix $S(\C,F)$ is a submatrix of the S-matrix $S({\D^{\Z/2N\Z}})$, namely for $C\in \O_\C^F, M\in \O_\M$ we have
\beq
S(\C,F)_{C,M}=\SD_{(C,\psi_C),(M,\psi_M)}.
\eeq
(ii) The entries of $S(\C,F)$ are cyclotomic integers. If $C\in \O_\C^F, M\in \O_\M$, then in fact  $\frac{S(\C,F)_{C,M}}{\dim C}$ and $\frac{S(\C,F)_{C,M}}{\dim_\M(M)}$ are cyclotomic integers.\\
(iii) In view of (ii), it is meaningful to take the complex conjugate transpose of $S(\C,F)$. Moreover, we have
\beq
S(\C,F)\cdot \overline{S(\C,F)}^T = \dim \C\cdot I = \overline{S(\C,F)}^T\cdot S(\C,F).
\eeq
(iv) For $C\in \O_\C^F, M\in \O_\M$, we have $(\psi_{C^*}\circ\psi_C^*)\cdot \overline{S(\C,F)}_{C,M}=S(\C,F)_{C^*,M}$.
\ethm

We will prove a more general version of this result in \S\ref{s:pmrg}. In the course of the proof, we will see that $\SD_{(C,\psi_C),(M,\psi_M)}$ does not depend on the choice of the isomorphism $\psi_M$. However it does depend on our original choice of isomorphism $\psi_C$.

As a corollary we obtain the following well known fact:
\bcor
The sets $\O_\C^F$ and $\O_\M$ have the same cardinality.
\ecor

\brk
Let $G$ be as in Remark \ref{r:connuni}. The statement above is an analogue of the fact that the number of irreducible representations of $G(\Fq)$ is equal to the number of Frobenius stable character sheaves on $G$. We refer to \cite{De1} for precise statements.
\erk
%\begin{thmbis}{t:main}\label{t:main'}
%\end{thmbis}

We now relate the crossed S-matrix $S(\C,F)$ to the character table of the algebra $K(\C,F)$ and derive an analogue of the Verlinde formula in our situation. Let $M\in \O_\M$. Let us define a $k$-linear map $\chi_M:K(\C,F)\to k$. For $C\in \O_\C^F$, define 
\beq
\chi_M([(C,\psi_C)])=\frac{S(\C,F)_{C,M}}{\dim_\M(M)}
\eeq
and extend $k$-linearly. We will prove the following as a consequence of Theorem \ref{t:main}:
\bthm\label{t:conseq}
(i) For each $M$ in $\O_\M$, $\chi_M$ is a character of the $k$-algebra $K(\C,F)$ i.e. $\chi_M$ is a ring homomorphism. The map $M\mapsto \chi_M$ defines a bijection between the set $\O_\M$ and the set of irreducible characters of the algebra $K(\C,F)$. Equivalently, the map $\chi:K(\C,F)\rar{}\Fun_k(\O_\M)$ defined by $\chi([(C,\psi_C)])(M)=\chi_M([(C,\psi_C)])$ is an isomorphism of $k$-algebras.\\
(ii) For $C\in \O_\C^F$, let $A_C$ denote the $\O_\C^F\times \O_\C^F$ matrix of left multiplication by $[(C,\psi_C)]$ in $K(\C,F)$ with respect to the basis $\{[(C',\psi_{C'})]\}_{C'\in \O_\C^F}$, i.e. $(A_C)_{D,C'}=a_{C,C'}^D$ (cf. (\ref{e:strucconst})). Let $\Delta_C$ be the $\O_\M\times \O_\M$ diagonal matrix with $(\Delta_C)_{M,M}=\frac{S(\C,F)_{C,M}}{\dim_\M(M)}=\frac{S(\C,F)_{C,M}}{S(\C,F)_{\un,M}}$. Then
\beq
(A_C)^T\cdot S(\C,F)=S(\C,F)\cdot \Delta_{C}.
\eeq
(iii) (An analogue of Verlinde formula.) The structure constants of the algebra $K(\C,F)$ can be described in terms of the crossed S-matrix as follows:
\beq
a_{C,C'}^D=\frac{1}{\dim \C}\sum\limits_{M\in \O_\M} \frac{S(\C,F)_{C,M}\cdot S(\C,F)_{C',M}\cdot\overline{S(\C,F)}_{D,M}}{\dim_\M(M)}.
\eeq
As a consequence we see that the numbers described by the right hand side above lie in $\Z[\omega]$.
\ethm

\brk
Note that if the modular autoequivalence $F$ is such that it induces an action of $\Z/2\Z$ on $\C$, then the structure constants $a^D_{C,C'}$ must in fact be integral. If $F=\id_\C$, then the structure constants are positive integral.   
\erk

\bcor\label{c:minid}
For $M\in \O_\M$, we have the delta function $\delta_M$ which is an idempotent in $\Fun_k(\O_\M)$. Let us set $e_M:=\chi^{-1}(\delta_M)\in K(\C,F)$, where $\chi$ is as defined in Theorem \ref{t:conseq}(i). Then the set of idempotents in $K(\C,F)$ is equal to $\{e_M|M\in \O_\M\}$. Moreover, we have
\beq
S(\C,F)\cdot\left(\begin{array}{c}\vdots\\\frac{e_M}{\dim_\M M}\\\vdots\end{array}\right)=\left(\begin{array}{c}\vdots\\{[(C,\psi_C)]}\\\vdots\end{array}\right), \mbox{ or equivalently}
\eeq
\beq
e_M=\frac{\dim_\M M}{\dim \C}\sum\limits_{C\in \O_\C^F}\overline{S(\C,F)}_{C,M}[(C,\psi_C)] \mbox{ for all $M\in \O_\M$ and hence}
\eeq
\beq
\l(e_M)=\frac{(\dim_\M M)^2}{\dim \C} \mbox{ for all $M\in \O_\M$},
\eeq
where $\l$ is the linear functional on the Frobenius algebra $K(\C,F)$ (see Definition \ref{d:linfun}).
\ecor
\bpf
From the definition of $\chi$, we see that 
\beq
\chi([(C,\psi_C)])=\sum\limits_{M\in \O_\M}\frac{S(\C,F)}{\dim_\M M}\cdot\delta_M. 
\eeq
Applying $\chi^{-1}$ we get the desired results.
\epf

\brk\label{r:wkresults}
All the results stated in this section are well known in the case when the modular autoequivalence $F$ is the identity on $\C$ (for example see \cite{BK, ENO1}). For example it is well known that the S-matrix of a modular category $\C$ is unitary (up to a suitable normalization) and that it can be thought of as the character table of the algebra $K_0(\C)\otimes k=K(\C,\id_C)$. We will in fact use these well known results to prove our results for general $F$.
\erk

\brk
By Theorem \ref{t:main}(ii), all $S(\C,F)_{C,M}\in \Qab$. In view of this, we may replace the field $k$ by its subfield $\Qab$ and the $k$-algebra $K(\C,F)$ by the $\Qab$-algebra $\KabC:=K_{\Z[\omega]}(\C,F)\otimes_{\Z[\omega]}\Qab$ in Theorem \ref{t:conseq} and Corollary \ref{c:minid}. 
\erk

\brk
So far we have  seen two bases of the space $K(\C,F)$: $\{[(C,\psi_C)]|C\in \O_\C^F\}$ and $\{\frac{e_M}{\dim_\M M}|M\in \O_\M\}$. In our analogy with the theory of character sheaves (see also Remark \ref{r:connuni}), the former is analogous to the trace functions of $F$-stable character sheaves and the latter to the irreducible characters (cf. \cite{De1}).
\erk

\section{Proofs of the main results and generalizations}\label{s:pmrg}
In this section we will prove the main results stated in \S\ref{s:ucsm}. In fact we will now state and prove some more general versions of these results. In particular for each $a\in \ZbN$ we will define crossed S-matrices $S(\M_a,F)$ and study their properties. If $a=0$, we recover the crossed S-matrix $S(\C,F)$ defined previously. In our analogy with the theory of character sheaves, the more general crossed S-matrices $S(\M_a,F)$ are closely related to the notion of Shintani descent of Frobenius stable characters. In fact in \S\ref{s:sdmc} we will define the notion of Shintani descent in our setting of modular categories.

\subsection{The Frobenius $*$-algebra $K_{\Qab}(\D,F)$}\label{s:fak}
The duality in a fusion category equips its Grothendieck ring with a $*$-ring structure. Hence $K_0(\D^{\Z/2N\Z})$ is a commutative $*$-ring. The field $\Qab$ is equipped with complex conjugation. Hence $K_0(\D^{\Z/2N\Z})\otimes \Qab$ is a $*$-algebra over $\Qab$. Let $\omega'$ be a primitive $2N$-th root of unity. The ideal in  $K_0(\D^{\Z/2N\Z})\otimes \Qab$ generated by $[(\un,\omega')]-\omega'$ is a $*$-ideal. Hence the $\Qab$-algebra $K_{\Qab}(\D,F):=\frac{K_0(\D^{\Z/2N\Z})\otimes \Qab}{\<[(\un,\omega')]-\omega'\>}$ is also a $*$-algebra. Note that the $\Qab$-algebra $K_{\Qab}(\C,F):=K_{\Z[\omega]}(\C,F)\otimes_{\Z[\omega]}\Qab$ (cf. \S\ref{s:sfa}) is a $*$-subalgebra of $K_{\Qab}(\D,F)$. Moreover using similar arguments as in \S\ref{s:faawma},\S\ref{s:sfa}, we see that $K_{\Qab}(\D,F)$ is in fact a Frobenius $*$-algebra.
 
\bdefn
We define a Hermitian form $\<\cdot,\cdot\>$ on $K_{\Qab}(\D,F)$ by setting
\beq
\<a_1,a_2\>=\lambda(a_1a_2^*)
\eeq
where $\l:\KabD\rar{}\Qab$ is the linear functional that defines the Frobenius structure on $\KabD$.
\edefn

Let $\O_\D$ denote the set of simple objects of $\D$. The autoequivalence $F:\D\rar{}\D$ induces a permutation of the set $\O_\D$. Recall that the action of $\Z/2N\Z$ on $\D$ in fact comes from an action of $\Z/N\Z$. As in \S\ref{s:faawma},\S\ref{s:sfa}, we can prove the following:
\bprop\label{p:de1gen}
For each $D\in \O_\D^F$ choose $\psi_D:F(D)\rar{\cong}D$ so that $(D,\psi_D)$ defines an object of $\D^{\Z/N\Z}\subset \D^{\Z/2N\Z}$. Then the set $\{[(D,\psi_D)]|D\in \O_\D^F\}$ is an orthonormal basis of $\KabD$ and hence $\<\cdot,\cdot\>$ is a positive definite Hermitian form on $\KabD$. The structure constants of the algebra with respect to this basis lie in the ring $\Z[\omega]\subset \Qab$.
\eprop

We have the following grading on the modular category $\D^{\Z/2N\Z}$:
\beq
\D^{\Z/2N\Z}=\bigoplus\limits_{a\in \Z/2N\Z}\M_a^{\Z/2N\Z} 
\eeq
and hence we obtain the corresponding grading of $\Qab$-algebras
\beq
\KabD=\bigoplus\limits_{a\in \Z/2N\Z}K_{\Qab}(\M_a,F).
\eeq

\subsection{The characters of the algebra $\KabD$}\label{s:cak}
In this section we state more general versions of some of the results from \S\ref{s:ucsm}. We will see that the characters of the algebra $\KabD$ are parametrized by the set $\O_{\M^{\Z/2N\Z}}$ of simple objects in the component $\M^{\Z/2N\Z}\subset \D^{\Z/2N\Z}$.

\blem\label{l:smblocks}
Let $m_1,m_2\in \Z/2N\Z$ and let $(M_i,\psi_i)\in \M_{m_i}^{\Z/2N\Z}$ where $\psi_i:F(M_i)\rar{\cong}M_i$ determines the $\Z/2N\Z$-equivariance structure. Let $\zeta_1,\zeta_2$ be any $2N$-th roots of unity. Then $(M_i,\zeta_i\psi_i)\in \M_{m_i}^{\Z/2N\Z}$. If  $(M_i,\psi_i)\in \M_{m_i}^{\Z/2N\Z}$ are simple objects, then we have the following relationship between entries of the S-matrix of $\Deq$:
\beq
\SD_{(M_1,\zeta_1\psi_1),(M_2,\zeta_2\psi_2)}=\zeta_1^{m_2}\zeta_2^{m_1}\cdot\SD_{(M_1,\psi_1),(M_2,\psi_2)}.
\eeq
\elem
\bpf
By the definition of the equivariantization and the S-matrix, the above entries of the S-matrix are given by the traces of the following compositions
\beq
M_1\otimes M_2\xto{\beta_{M_1,M_2}}F^{m_1}(M_2)\otimes M_1\xto{\psi_2^{(m_1)}}M_2\otimes M_1\xto{\beta_{M_2,M_1}}F^{m_2}(M_1)\otimes M_2\xto{\psi_1^{(m_2)}}M_1\otimes M_2
\eeq
\beq
M_1\otimes M_2\xto{\beta_{M_1,M_2}}F^{m_1}(M_2)\otimes M_1\xto{\zeta_2^{m_1}\psi_2^{(m_1)}}M_2\otimes M_1\xto{\beta_{M_2,M_1}}F^{m_2}(M_1)\otimes M_2\xto{\zeta_1^{m_2}\psi_1^{(m_2)}}M_1\otimes M_2.
\eeq
Hence the lemma follows.
\epf

\bcor\label{c:0insd}
Suppose $M\in \O_\M=\O_\M^F$ and let $(M,\psi'_M)\in \M^{\Z/2N\Z}$. Let $(L,\psi'_L)\in \O_{\Deq}$. Suppose that $L\cong \bigoplus L_i$ as objects in $\D$ with $L_i\in \O_\D$. Since $(L,\psi'_L)\in \Deq$ is simple, the $L_i$ must be pairwise non-isomorphic and must form one single orbit $O$ for the $\Z/2N\Z$-action on $\O_\D$. If $O$ is not singleton, then we must have
\beq
\SD_{(L,\psi'_L),(M,\psi'_M)}=0.
\eeq
\ecor
\bpf
Suppose that $|O|=n>1$. Then, $n|N$ since the action of $\Z/2N\Z$ in fact comes from an action of $\Z/N\Z$ on $\O_\D$. Let $\zeta\neq 1$ be an $n$-th root of unity. Then it is easy to check that $(L,\zeta\psi'_L)\cong (L,\psi'_L)$ as objects of $\Deq$ (cf. proof of Proposition \ref{p:de1frobalg}). On the other hand, by Lemma \ref{l:smblocks} we have 
\beq
\SD_{(L,\zeta\psi'_L),(M,\psi'_M)}=\zeta\cdot\SD_{(L,\psi'_L),(M,\psi'_M)}.
\eeq
Hence we must have $\SD_{(L,\psi'_L),(M,\psi'_M)}=0.$
\epf

It is well known that the entries of the S-matrix $\SD$ lie in $\Qab$ (cf. \cite[\S10]{ENO1}). From this we deduce:
\bcor\label{c:kabdchars}
For each $D\in \O_\D^F$, choose $\psi_D$ as in Proposition \ref{p:de1gen} so that $\{[(D,\psi_D)]|D\in \O_\D^F\}$ is an orthonormal basis of $\KabD$. Let $(M,\psi'_M)$ be any simple object of $\M^{\Z/2N\Z}\subset \Deq$. Then the $\Qab$-linear map $\chi_{M,\psi'_M}:\KabD\rar{}\Qab$ defined by $\chi_{M,\psi'_M}([(D,\psi_D)])=\frac{\SD_{(D,\psi_D),(M,\psi'_M)}}{\dim_\M M}$ is in fact a 1-dimensional character. This establishes a bijection between the set $\O_{\M^{\Z/2N\Z}}$ and the set of irreducible characters of the algebra $\KabD$.
\ecor
\bpf
By \cite[Thm. 3.1.11]{BK} all the characters of the algebra $K_0(\Deq)\otimes \Qab$ are given by 
\beq
\chi_{L,\psi'_L}:[(D,\psi'_D)]\mapsto \frac{\SD_{(D,\psi'_D),(L,\psi'_L)}}{\dim_{\Deq} (L,\psi'_L)}=\frac{\SD_{(D,\psi'_D),(L,\psi'_L)}}{\SD_{(\un,\id_\un),(L,\psi'_L)}}
\eeq
where $(D,\psi'_D),(L,\psi'_L)$ are simple objects of $\Deq$. Then by Lemma \ref{l:smblocks} we see that $\chi_{L,\psi'_L}([(\un,\omega')])=\omega'$ if and only if $(L,\psi'_L)\in \M^{\Z/2N\Z}$, i.e. the character $\chi_{L,\psi'_L}$ of $K_0(\Deq)\otimes \Qab$ factors through the quotient $\KabD$ if and only if $(L,\psi'_L)\in \M^{\Z/2N\Z}$. This gives us the desired result.
\epf

\subsection{The matrices $S(\M_a,F)$}\label{s:msmf}
In this section we describe some other interesting matrices that occur as submatrices in the S-matrix of $\Deq$, namely for each $a\in \Z/2N\Z$ we will define the crossed S-matrix $S(\M_a,F)$ with rows parametrized by $\O^F_{\M_a}$ and columns parametrized by $\O_\M$. In general the matrix $S(\M_a,F)$ depends on certain choices and is only well defined up to rescaling rows and columns by roots of unity. We make this more precise in Remark \ref{r:choices} below. 
The choices that we need to make are:\\
(i) For each $L\in \O_{\M_a}^F$, we choose a lift $(L,\psi_L)\in \M_a^{\Z/N\Z}\subset \Deq$. \\
(ii) For each $M\in \O_{\M}$, we choose a lift $(M,\psi_M)\in \M^{\Z/N\Z}\subset \Deq$.
\bdefn
The crossed S-matrix $S(\M_a,F)$ (relative to the above choices) is the $\O_{\M_a}^F\times \O_\M$ submatrix of $\SD$ whose entries are given by
\beq
S(\M_a,F)_{L,M}=\SD_{(L,\psi_L),(M,\psi_M)}.
\eeq
\edefn

By \cite[3.1.21]{BK} and \cite[\S10]{ENO1} we see that
\blem\label{l:smaf}
The entries of the matrix $S(\M_a,F)$ are cyclotomic integers. In fact the numbers $\frac{S(\M_a,F)_{L,M}}{\dim_{\M_a}L}$ and $\frac{S(\M_a,F)_{L,M}}{\dim_{\M}M}$ are cyclotomic integers.
\elem

\brk\label{r:choices}
Let $\zeta$ be any $N$-th root of unity. If we rescale some $\psi_L$ for $L\in \O_{\M_a}^F$ by $\zeta$, then by Lemma \ref{l:smblocks} the $L$-th row of $S(\M_a,F)$ gets scaled by $\zeta$. On the other hand, if we rescale some $\psi_M$ for $M\in \O_\M$ by $\zeta$, then the $M$-th column of $S(\M_a,F)$ gets rescaled by $\zeta^a$. Note that for $a=0$, we recover the crossed S-matrix $S(\C,F)$ and in this case the different choices of the $\psi_M$ do not affect the matrix $S(\C,F)$.
\erk

\subsection{Unitarity of the matrices $S(\M_a,F)$}\label{s:umsmf}
Let us suppose that for each $D\in \O_\D^F=\coprod\limits_{a\in \ZbN}\O_{\M_a}^F$ we have chosen a lift $(D,\psi_D)\in \D^{\ZN}\subset \Deq$. With this choice, for each $a\in \ZbN$, we have the matrix $S(\M_a,F)$. By Lemma \ref{l:smaf}, it makes sense to take complex conjugates of entries of $S(\M_a,F)$. We will now prove the following generalization of Theorem \ref{t:main}(iii), (iv):
\bthm\label{t:genmain}
We have
\beq
S(\M_a,F)\cdot \overline{S(\M_a,F)}^T = \dim \C\cdot I = \overline{S(\M_a,F)}^T\cdot S(\M_a,F).
\eeq\label{e:smconj}
For $L\in \O_{\M_a}^F, M\in \O_\M$, we have 
\beq
(\psi_{L^*}\circ\psi_L^*)\cdot \overline{S(\M_a,F)}_{L,M}=S(\M_{-a},F)_{L^*,M}.
\eeq
\ethm
\bpf
Since $\SD$ is the S-matrix of a modular category we know (cf. \cite{BK}, \cite{ENO1},\cite{DGNO}) that 
\beq\label{e:sdunitary}
\SD\overline{\SD}^T=\overline{\SD}^T\SD=\dim \Deq\cdot I=4N^2\dim\C\cdot I.
\eeq
Note that for each $L\in \O_{\M_a}^F, M\in \O_\M$ we have chosen isomorphisms $\psi_L, \psi_M$. In fact let us choose $\Z/N\Z$-equivariance isomorphisms $\psi_D$ for each $D\in \O_\D^F$. Let us look at the columns in $\SD$ corresponding to the simple objects $(M,\psi'_M)\in \O_{\M^{\Z/2N\Z}}$. By Lemma \ref{l:smblocks} the submatrix formed by these columns is of the following block form
\beq
S:=\left(\begin{array}{cccc}
A_0&A_0&\cdots&A_0\\
A_1&\omega'A_1 & &\omega'^{2N-1}A_1\\
 \vdots&&&\\
A_{2N-1}&\omega'^{2N-1}A_{2N-1}&\cdots&\omega'^{(2N-1)^2}A_{2N-1}
\end{array}\right),
\eeq
where the first column of blocks corresponds to those simple objects in $\M^{\Z/2N\Z}$ that are defined by our chosen $\Z/N\Z$-equivariance structure for $M\in \O_\M$ and the subsequent columns of blocks correspond to the $\Z/2N\Z$-equivariance structures obtained by multiplying our chosen ones by various $2N$-th roots of unity. The rows in this block matrix correspond to the simple objects in the subcategories $\M_a^{\Z/2N\Z}\subset \Deq$ for $a\in \Z/2N\Z$. Then by (\ref{e:sdunitary}) we have $\overline{S}^TS=4N^2\dim\C\cdot I$ which gives us that 
\beq
\sum\limits_{a\in \Z/2N\Z}\overline{A_a}^TA_a=4N^2\dim\C\cdot I \mbox{ and }
\eeq
\beq
\sum\limits_{a\in \Z/2N\Z}\omega'^{ab}\overline{A_a}^TA_a=0 \mbox{ for $b\neq 0$ in $\Z/2N\Z$.}
\eeq
Using this we deduce that 
\beq\label{e:Aa}
\overline{A_a}^TA_a=2N\dim\C\cdot I \mbox{ for each $a\in \Z/2N\Z$.}
\eeq
Now let us zoom in and see what each $A_a$ looks like. Again by Lemma \ref{l:smblocks} and Corollary \ref{c:0insd}, we obtain the following block form
\beq
A_a=\left(\begin{array}{c}
S(\M_a,F)\\
\omega'S(\M_a,F)\\
\vdots\\
\omega'^{2N-1}S(\M_a,F)\\
0
\end{array}\right),
\eeq
where the first block corresponds to the simple objects in $\M_a^{\ZbN}$ of the form $(L,\psi_L)$ with $L\in \O_{\M_a}^F$ and $\psi_L$ the chosen $\ZbN$-equivariance structure. For $a\in \ZbN$, $a$-th block corresponds to the simple objects in $\M_a^{\ZbN}$ of the form $(L,\omega'^a\psi_L)$ with $L\in \O_{\M_a}^F$. Finally the last block corresponds to the remaining simple objects in $\M_a^{\ZbN}$. This last block is equal to zero by Corollary \ref{c:0insd}. Combining this with (\ref{e:Aa}), we conclude that $ \overline{S(\M_a,F)}^T\cdot S(\M_a,F)=\dim \C\cdot I$.

Now let us look at those columns in $\SD$ that correspond to simple objects of the form ${(L,\psi'_L)}\in \M_a^{\ZbN}\subset \Deq$ where $L\in \O_{\M_a}^F$. Now we repeat the same argument as before. As before, the submatrix $S'$ formed by these columns is of the form  
\beq
S'=\left(\begin{array}{cccc}
A'_0&A'_0&\cdots&A'_0\\
A'_1&\omega'A'_1 & &\omega'^{2N-1}A'_1\\
 \vdots&&&\\
A'_{2N-1}&\omega'^{2N-1}A'_{2N-1}&\cdots&\omega'^{(2N-1)^2}A'_{2N-1}
\end{array}\right)
\eeq
and from this we once again deduce that 
\beq\label{e:A'a}
\overline{A'_a}^TA'_a=2N\dim\C\cdot I \mbox{ for each $a\in \Z/2N\Z$.}
\eeq
Now let us zoom into the matrix $A'_1$. By the symmetry of the matrix $\SD$ and Lemma \ref{l:smblocks}, we see that it is of the form
\beq
A'_1=\left(\begin{array}{c}
S(\M_a,F)^T\\
\omega'^aS(\M_a,F)^T\\
\vdots\\
\omega'^{a(2N-1)}S(\M_a,F)^T
\end{array}\right)
\eeq
and from this we conclude that $S(\M_a,F)\cdot \overline{S(\M_a,F)}^T = \dim \C\cdot I$ as desired.

To prove (\ref{e:smconj}) we have 
$$(\psi_{L^*}\circ\psi_L^*)\cdot \overline{S(\M_a,F)}_{L,M}=(\psi_{L^*}\circ\psi_L^*)\cdot\overline{S(\Deq,F)}_{(L,\psi_L),(M,\psi_M)}$$
$$=(\psi_{L^*}\circ\psi_L^*)\cdot{S(\Deq,F)}_{(L^*,(\psi^*_L)^{-1}),(M,\psi_M)}$$
$$={S(\Deq,F)}_{(L^*,\psi_L^*),(M,\psi_M)}$$
$$={S(\M_{-a},F)}_{L^*,M}$$
as desired.
\epf
As a consequence, we obtain
\bcor\label{c:eqcar}
The sets $\O_\M$ and $\O_{\M_a}^F$ have equal cardinality.
\ecor

\brk 
Let $G$ be any connected algebraic group over $\Fqcl$ with an $\Fq$-Frobenius $F$. The above statement is an analogue of the fact that the number of $F$-stable irreducible representations of $G(\F_{q^m})$ is equal to the number of irreducible representations of $G(\Fq)$.
\erk

\subsection{Completion of the proofs}\label{s:cp}
We now complete the proofs of our main results stated in \S\ref{s:ucsm}. Theorem \ref{t:main}(i) follows immediately from the definition of equivariantization and that of the S-matrix (see also proof of Lemma \ref{l:smblocks}). We have already proved more general versions (Lemma \ref{l:smaf} and Theorem \ref{t:genmain}) of the remainder of Theorem \ref{t:main}.

Let us now prove Theorem \ref{t:conseq}. By Corollary \ref{c:kabdchars} the characters of the algebra $\KabD$ are precisely those of the form $\chi_{M,\psi'_M}$ where $(M,\psi'_M)\in \O_{\M^{\ZbN}}$. By Theorem \ref{t:main}(i) the restriction of $\chi_{M,\psi'_M}$ to the subalgebra ${\KabC}\subset \KabD$ agrees with the map $\chi_M$ from Theorem \ref{t:conseq}. This proves that $\chi_M$ is a character. The invertibility of the matrix $S(\C,F)$ implies that the characters $\chi_M$ for distinct $M\in \O_\M$ are distinct. Moreover, since $|\O_\M|=|\O_\C^F|=\dim_{\Qab}(\KabC)$, these must be all the characters. This completes the proof of Theorem \ref{t:conseq}(i).

Theorem \ref{t:conseq}(ii) readily follows from the definition of $\chi_M$ and the equality 
\beq
\chi_M\left(\sum\limits_{D\in \O_\C^F}a^D_{C,C'}[(D,\psi_D)]\right)=\chi_M([(C,\psi_C)]\cdot[(C',\psi_{C'})])=\chi_M([(C,\psi_C)])\cdot\chi_M([(C',\psi_{C'})]).
\eeq
From Theorem \ref{t:conseq}(ii) and Theorem \ref{t:main}(iii), we deduce that $(A_C)^T=S(\C,F)\cdot\Delta_C\cdot\frac{\overline{S(\C,F)}^T}{\dim\C}$ which is equivalent to Theorem \ref{t:conseq}(iii). This completes the proofs of all the results stated in \S\ref{s:ucsm}.

\section{Shintani descent for modular categories}\label{s:sdmc}
We have seen that we have the equality $|\O_\C^F|=|\O_\M|=\dim_kK(\C,F)$. In fact  we have also proved (see Corollary \ref{c:eqcar}) that $|\O_{\M_a}^F|=\dim_kK(\C,F)$ for each $a\in \Z/2N\Z$. In this section we will define the notion of Shintani descent in the setting of modular categories. This is a categorical analogue of the classical notion of Shintani descent that appears in the character theory of algebraic groups defined over finite fields (cf. \cite{S}, \cite{De2}).  Let $m$ be a positive integer. We will define the $m$-th Shintani descent map (it is only well defined up to scaling by $m$-th roots of unity)
\beq
\Sh_m:\O_{\M_m}^F\hookrightarrow K_{\Qab}(\C,F)=K_{\Z[\omega]}(\C,F)\otimes_{\Z[\omega]}\Qab
\eeq
such that the image $\Sh_m(\O_{\M_m}^F)$ is an orthonormal (see \S\ref{s:fak}) basis in $K_{\Qab}(\C,F)$.

\subsection{Twists in spherical braided crossed categories}\label{s:tsbcc}
To define Shintani descent let us first introduce the notion of twists in braided crossed categories. Recall that the notion of a spherical structure on a braided fusion category is equivalent to that of a ribbon twist structure (cf. \cite{DGNO}). Let us generalize this notion to the setting of braided crossed categories. Let $\Gamma$ be a finite group and let $\D$ be a braided $\Gamma$-crossed  fusion category
\beq
\D=\bigoplus\limits_{\g\in \Gamma}\M_\g.
\eeq 
Recall that we also have a monoidal action of $\Gamma$ on $\D$ and furthermore, for each $\g\in \Gamma$ we have $\g(\M_\g)\subset \M_\g$ and that the functor $\g|_{\M_\g}$ is isomorphic to $\id_{\M_\g}$ as a functor between abelian categories.  A twist in $\D$ is a collection of twists $\theta^\g$ for each $\g\in \Gamma$, where each $\theta^\g:\id_{\M_\g}\rar{\cong}\g|_{\M_\g}$ is a natural isomorphism satisfying the following compatibility relations: If $\ga,\gb\in \Gamma$, $M_1\in \M_{\ga}$ and $M_2\in \M_{\gb}$ then\\ 
(i) we should have the equality $\ga(\theta^{\gb}_{M_2})=\theta^{\ga\gb\ga^{-1}}_{\ga(M_2)}$ and\\
(ii) the twist $\theta^{\ga\gb}_{M_1\otimes M_2}$ should equal the composition
{\scriptsize
\beq
M_1\otimes M_2\xto{\beta_{M_1,M_2}}\ga(M_2)\otimes M_1\xto{\beta_{\ga(M_2),M_1}}\ga\gb\ga^{-1}(M_1)\otimes \ga(M_2)\xto{{\ga\gb\left(\theta^\ga_{\ga^{-1}(M_1)}\right)\otimes \ga\left(\theta^\gb_{M_2}\right)}}\ga\gb(M_1)\otimes \ga\gb(M_2).
\eeq
}

Then we have a 1-1 correspondence between twists on $\D$ in this sense and pivotal structures on $\D$ compatible with the action of $\Gamma$. Namely given a twist on $\D$, $\g\in \Gamma$ and $M\in \M_\g$ we have the pivotal structure on $\D$ defined using the isomorphisms
\beq\label{e:twtopiv}
M\rar{}M^*\otimes M^{**}\otimes M\rar{}M^*\otimes \g(M)\otimes M^{**}\xto{\id_M^*\otimes(\theta^\g_{M})^{-1}\otimes \id_{M^{**}}}M^*\otimes M\otimes M^{**}\rar{}M^{**}
\eeq
where the morphisms are defined using the coevaluations, crossed braidings and evaluations. On the other hand, given a pivotal structure we can define the twist $\theta^\g_M$ as the composition
\beq\label{e:pivtotw}
M\rar{}M\otimes M\otimes M^*\rar{}\g(M)\otimes M\otimes M^*\rar{}\g(M)\otimes M^{**}\otimes M^*\rar{}\g(M),
\eeq
where the morphisms are defined using the coevaluations, crossed braidings, pivotal structure and evaluations.
\brk
It is convenient to visualize these maps using ``string diagrams''. Using this visual aid, it is straightforward to check that the compositions (\ref{e:twtopiv}) and (\ref{e:pivtotw}) are isomorphisms (by drawing the string diagrams corresponding to the inverse morphisms) and that this indeed establishes an equivalence between the notions of twists in $\D$ and pivotal structures on $\D$ compatible with the $\Gamma$-action.
\erk

Moreover, the pivotal structure corresponding to a twist $\theta$ in $\D$ is spherical if in addition to (i) and (ii)\\
(iii) we have the equality $(\theta^\g_M)^*=\theta^{(\g^{-1})}_{\g(M^*)}$.

\brk\label{r:twind}
We will use this construction for the spherical braided $\Z/2N\Z$-crossed category $\D$ introduced in \S\ref{s:csbcc}. Hence we can construct a twist in the braided crossed category $\D$ agreeing with the ribbon twist in the modular category $\C$ corresponding the spherical structure on $\D$. In other words, for each $a\in \Z/2N\Z$, $L\in \M_a$, we have natural isomorphisms $\theta^a_L:L\rar{\cong}F^a(L)$.
\erk

\subsection{Shintani matrices}\label{s:sm}
Let $m$ be a positive integer. By a slight abuse of notation, we will also use $m$ to denote its reduction modulo $2N$. Recall that we have defined the $\O_{\M_m}^F\times \O_\M$ matrix $S(\M_m,F)$ in \S\ref{s:msmf}. To do this, for each $L\in \O_{\M_m}^F, M\in \O_\M$ we chose lifts $(L,\psi_L)\in \M_m^{\ZN}, (M,\psi_M)\in \M^{\ZN}$. In this section we will define another $\O_{\M_m}^F\times \O_\M$ matrix $\Sh_m(\M_m,F)$ known as the $m$-th Shintani matrix which is closely related to $S(\M_m,F)$.

For each $L\in {\M_m}, M\in \M$ we have the twists
\beq
\theta^m_L:L\rar{\cong}F^m(L),
\eeq
\beq
\theta^1_M:M\rar{\cong}F(M).
\eeq
Now for each $L\in \O_{\M_m}^F$ we choose an isomorphism $\eta_L:F(L)\rar{\cong} L$ such that the composition
\beq\label{e:etacondi}
L\xto{\theta^m_L}F^m(L)\xto{F^{m-1}(\eta_L)}F^{m-1}(L)\rar{}\cdots\rar{}F(L)\xto{\eta_L}L
\eeq
is equal to $\id_L$. There are $m$ choices for $\eta_L$ which differ from each other by $m$-th roots of unity. Also for each $M\in \O_\M$, let $\eta^m_M$ denote the composition
\beq
\eta_M^m:F^m(M)\xto{F^{m-1}((\theta^{1}_M)^{-1})}F^{m-1}(M)\rar{}\cdots\rar{}F(M)\xto{(\theta^{1}_M)^{-1}}M.
\eeq 
 Then let us define the automorphism $\g_{L,\eta_L,M,\eta^m_M}$ in $\D$ as the composition
\beq\label{e:geta}
\g_{L,\eta_L,M,\eta^m_M}:L\otimes M\xto{\beta_{L,M}}F^m(M)\otimes L\xto{\beta_{F^m(M),L}}F(L)\otimes F^m(M)\xto{\eta_L\otimes \eta^m_M}L\otimes M.
\eeq
\bdefn
The $m$-th Shintani matrix $\Sh_m(\M_m,F)$ is the $\O_{\M_m}^F\times \O_\M$ matrix defined by
\beq
\Sh_m(\M_m,F)_{L,M}=\tr_{\M_{m+1}}(\g_{L,\eta_L,M,\eta^m_M}) \mbox{ for $L\in \O_{\M_a}^F, M\in \O_\M$.}
\eeq
\edefn

\brk
The different choices of $\eta_L$ amount to rescaling the rows of $\Sh_m(\M_m,F)$ by $m$-th roots of unity.
\erk

\brk
The definition of the Shintani matrices is motivated by \cite[Thm. 5.4, Cor. 5.5]{De2}.
\erk

\subsection{Relationship with the matrices $S(\M_m,F)$}\label{s:rwmsmf}
In this section, we will see that the two $\O_{\M_m}^F\times \O_\M$ matrices $\Sh_m(\M_m,F)$ and $S(\M_m,F)$ are closely related.

As in the definition of $S(\M_m,F)$, we assume that we have chosen lifts $(D,\psi_D)\in \D^{\ZN}$ for each $D\in\O_\D^F$. Then for $L\in \O_{\M_m}^F, M\in \O_\M$, $S(\M_m,F)_{L,M}$ is defined as the trace of (see also proof of Lemma \ref{l:smblocks})
\beq\label{e:gpsi}
\g_{L,\psi_L,M,\psi^{(m)}_M}:L\otimes M\xto{\beta_{L,M}}F^m(M)\otimes L\xto{\beta_{F^m(M),L}}F(L)\otimes F^m(M)\xto{\psi_L\otimes \psi^{(m)}_M}L\otimes M.
\eeq

By the definition of the equivariantization and the definition of twists in modular categories (for example in terms of string diagrams) we deduce
\blem
Let $(L,\psi'_L)$ be any object of $\M_m^{\ZbN}\subset \Deq$. Then the twist $\theta_{L,\psi'_L}$ in the modular category $\Deq$ is equal to the composition
\beq
\theta_{L,\psi'_L}:L\xto{\theta^m_L}F^m(L)\xto{{\psi'_L}^{(m)}}L.
\eeq
\elem\label{l:twindeq}
Comparing (\ref{e:geta}) and (\ref{e:gpsi}) we deduce that for $L\in \O_{\M_a}^F, M\in \O_\M$ we have
\beq
\Sh_m(\M_m,F)_{L,M}=(\eta_L^{-1}\circ\psi_L)\cdot({(\eta_M^m)}^{-1}\circ \psi^{(m)}_M)\cdot S(\M_m,F).
\eeq
Furthermore, by the definition of $\eta^m_M$ and Lemma \ref{l:twindeq} we deduce that
\beq
\Sh_m(\M_m,F)_{L,M}=(\theta'_{L,\psi_L,\eta_L})\cdot(\theta_{M,\psi_M})^m\cdot S(\M_m,F),
\eeq
where $\theta_{M,\psi_M}$ is the twist in $\D$ associated with the simple object $(M,\psi_M)$ and $\theta'_{L,\psi_L,\eta_L}:=\eta_L^{-1}\circ\psi_L$. Moreover by Lemma \ref{l:twindeq} and (\ref{e:etacondi}) we have $(\theta'_{L,\psi_L,\eta_L})^m=\theta_{L,\psi_L}$. Now twists in modular categories are always roots of unity. Hence $(\theta'_{L,\psi_L,\eta_L})$ must also be a root of unity. 

Hence we have proved:
\bprop\label{p:shandsmf}
Let $T(\M,F)$ be the $\O_\M\times \O_\M$ diagonal matrix such that $T(\M,F)_{M,M}=\theta_{M,\psi_M}$ (the twist in $\Deq$) for $M\in \O_\M$ and our chosen $\psi_M$. Let $T'(\M_m,F)$ be the $\O_{\M_m}^F\times \O_{\M_m}^F$ diagonal matrix such that $T'(\M_m,F)_{L,L}=\theta'_{L,\psi_L,\eta_L}$ for $L\in \O_{\M_m}^F$. Then all the entries of the diagonal matrices $T(\M,F)$ and $T'(\M_m,F)$ are roots of unity and we have
\beq
\Sh_m(\M_m,F)=T'(\M_m,F)\cdot S(\M_m,F)\cdot T(\M,F)^m.
\eeq
Hence the entries of the matrix $\Sh_m(\M_m,F)$ are cyclotomic integers and we have
\beq
\Sh_m(\M_m,F)\cdot\overline{\Sh_m(\M_m,F)}^T=\dim\C\cdot I=\overline{\Sh_m(\M_m,F)}^T\cdot \Sh_m(\M_m,F)
\eeq
\eprop
\brk
The rows of the matrix $\Sh_m(\M_m,F)$ are well defined up to scaling by $m$-th roots of unity. This is not true for the matrix $S(\M_m,F)$, where different choices of the isomorphisms $\psi_M$ may lead to rescaling of columns. 
\erk

\subsection{Shintani descent}\label{s:sd}
Let $m$ be any positive integer. We will now define the $m$-th Shintani descent map as a composition
\beq
\Sh_m:\O_{\M_m}^F\hookrightarrow \Fun_{\Qab}(\O_\M)\xto{\chi^{-1}}\KabC,
\eeq
where $\chi^{-1}$ is the inverse of the isomorphism $\chi$ defined in Theorem \ref{t:conseq}(i). For $L\in \O_{\M_m}^F$ we define the function $\chi_m(L)\in\Fun_{\Qab}(\O_\M)$ by setting 
\beq
\chi_m(L)(M):=\frac{\Sh_m(\M_m,F)_{L,M}}{\dim_\M(M)} \mbox{ for $M\in \O_\M$}
\eeq
and then we define $\Sh_m(L):=\chi^{-1}(\chi_m(L))\in \KabC$.
\brk
A different choice of the isomorphism $\eta_L$ amounts to rescaling $\Sh_m(L)$ by an $m$-th root of unity.
\erk
As a corollary of our results, we obtain
\bthm\label{t:sdmain}
(i) For each positive integer $m$, the image $\Sh_m(\O_{\M_m}^F)\subset \KabC$ is an orthonormal basis. We call this basis as the $m$-th Shintani basis of $\KabC$.\\
(ii) There exists a positive integer $m_0\in N\Z$ such that $T(\M,F)^{m_0}=I$. Then the $m$-th Shintani basis (up to rescaling by roots of unity) only depends on the residue class of $m$ modulo $m_0$. \\
(iii) If $m\in m_0\Z$, then the $m$-th Shintani basis of $\KabC$ coincides with the basis $\{[(C,\psi_C)]|C\in \O_\C^F\}$.
\ethm
\bpf
Recall from Corollary \ref{c:minid} that the set of idempotents $\{e_M|M\in \O_\M\}$ form a basis of $\KabC$. By definition, we have
\beq
\Sh_m(\M_m,F)\cdot\left(\begin{array}{c}\vdots\\\frac{e_M}{\dim_\M M}\\\vdots\end{array}\right)=\left(\begin{array}{c}\vdots\\{\Sh_m(L)}\\\vdots\end{array}\right).
\eeq
Then by Proposition \ref{p:shandsmf} and Corollary \ref{c:minid} we deduce statement (i). 

The existence of $m_0$ from (ii) follows from the fact that the matrix $T(\M,F)$ is diagonal with diagonal entries being roots of unity. Now the matrix $S(\M_m,F)$ only depends (up to $\pm1$) on the residue class of $m$ modulo $N$. By the condition on $m_0$, the matrix $T(\M,F)^m$ depends only on the residue class of $m$ modulo $m_0$. Hence by Proposition \ref{p:shandsmf}, the matrix $\Sh_m(\M_m,F)$ only depends (up to rescaling rows by roots of unity) on the residue of $m$ modulo $m_0$ which proves statement (ii). 

Now suppose that $m\in m_0\Z\subset N\Z$. Then $S(\M_m,F)=\pm S(\C,F)$ and $T(\C,F)^m=I$. Statement (iii) now follows from Proposition \ref{p:shandsmf} and Corollary \ref{c:minid}. 
\epf

\brk
Theorem \ref{t:sdmain} is a categorical analogue of \cite[Thm. 1.7(i),(iii), Def. 3.2]{De2}.
\erk

\subsection{The twisting operator and Shintani descent}\label{s:tosd}
In this section, we study the (unitary) twisting operator $\Theta:\KabC\rar{\cong}\KabC$ and relate it to Shintani descent for $m=1$.

We define the operator $\Theta:\KabC\rar{\cong}\KabC$ by setting $\Theta([(C,\psi)])=\theta_C\cdot[(C,\psi)]$ for $C\in \O_\C^F$, where $\theta_C$ is the twist in the modular category $\C$ and it equals the twist $\theta_{C,\psi_C}$ in the modular category $\Deq$. On the other hand, we have the 1st Shintani descent map $\Sh_1:\O_\M\hookrightarrow\KabC$ whose image is an orthonormal basis of $\KabC$. Note that the map $\Sh_1$ is uniquely determined, i.e. it does not depend on any choices since for each $M\in \O_\M$, $\eta_M$ must equal $(\theta^1_M)^{-1}$ (see (\ref{e:etacondi})). 

We have another basis of $\KabC$ parametrized by the set $\O_\M$, namely the basis $\{\frac{e_M}{\dim_\M M}|M\in \O_\M\}$. We now compare this basis with the 1st Shintani basis of $\KabC$.
\bthm\label{t:asai}
For $M\in \O_\M$, we have 
\beq
\Sh_1(M)=\tau^+(\C)\cdot\Theta^{-1}\left(\frac{e_M}{\dim_\M M}\right)
\eeq
where $\tau^+(\C):=\sum\limits_{C\in \O_\C}\theta_C\cdot (\dim C)^2$ is known as the Gauss sum of the modular category $\C$.
\ethm
\bpf
Let $T(\C,F)$ be the $\O_\C^F\times \O_\C^F$ diagonal matrix whose diagonal entries are $T(\C,F)_{C,C}=\theta_C$ for $C\in \O_\C^F$. Then by definition, we have
\beq
\Theta\left(\begin{array}{c}\vdots\\{[(C,\psi_C)]}\\\vdots\end{array}\right)=T(\C,F)\cdot\left(\begin{array}{c}\vdots\\{[(C,\psi_C)]}\\\vdots\end{array}\right).
\eeq
Now by Corollary \ref{c:minid}, we get
\beq
\Theta\left(\begin{array}{c}\vdots\\{\frac{e_M}{\dim_\M M}}\\\vdots\end{array}\right)=S(\C,F)^{-1}\cdot T(\C,F)\cdot S(\C,F)\cdot\left(\begin{array}{c}\vdots\\{\frac{e_M}{\dim_\M M}}\\\vdots\end{array}\right) \mbox{ and hence}
\eeq
\beq
\Theta^{-1}\left(\begin{array}{c}\vdots\\{\frac{e_M}{\dim_\M M}}\\\vdots\end{array}\right)=\frac{\overline{S(\C,F)}^{T}}{\dim\C}\cdot T(\C,F)^{-1}\cdot S(\C,F)\cdot\left(\begin{array}{c}\vdots\\{\frac{e_M}{\dim_\M M}}\\\vdots\end{array}\right).
\eeq
Hence to complete the proof, we must prove that
\beq
\Sh_1(\M,F)=\frac{\tau^+(\C)}{\dim\C}\cdot{\overline{S(\C,F)}^{T}}\cdot T(\C,F)^{-1}\cdot S(\C,F) \mbox{ or equivalently (see Prop. \ref{p:shandsmf}) that}
\eeq
\beq\label{e:tpt}
T(\M,F)\cdot S(\M,F)\cdot T(\M,F)=\frac{\tau^+(\C)}{\dim\C}\cdot{\overline{S(\C,F)}^{T}}\cdot T(\C,F)^{-1}\cdot S(\C,F). 
\eeq
To prove this last equality, we will use \cite[Thm. 3.1.7 (3.1.10)]{BK} for the modular category $\Deq$ which gives us $\left(S(\Deq)T(\Deq)\right)^3=\tau^+(\Deq)S(\Deq)^2$ which is equivalent to
\beq\label{e:sdrel}
T(\Deq)\cdot S(\Deq)\cdot T(\Deq)=\frac{\tau^+(\Deq)}{\dim\Deq}\cdot{\overline{S(\Deq)}^{T}}\cdot T(\Deq)^{-1}\cdot S(\Deq).
\eeq
Note that the matrix $T(\Deq)$ above is the $T$-matrix of the modular category $\Deq$, a diagonal matrix whose diagonal entries are the twists of the simple objects of $\Deq$.

We now follow the strategy of the proof of Theorem \ref{t:genmain} and look at the block decomposition of the matrices $\SD$ and $T(\Deq)$. For ease of notation, for each $a\in \ZbN$ we set $S_a:=S(\M_a,F)$ and $T_a=T(\M_a,F)$, i.e. $T_a$ is the $\O_{\M_a}^F\times \O_{\M_a}^F$ diagonal matrix such that $({T_a})_{L,L}=\theta_{L,\psi_L}$ for $L\in \O_{\M_a}^F$.

Consider the columns in $\SD$ corresponding to the simple objects $(M,\psi_M)\in \Deq$ for $M\in \O_\M$. We have seen the block decomposition of the submatrix  formed by these columns in the proof of Theorem \ref{t:genmain}. Now let us look at the same columns in the product $T(\Deq)^{-1}\cdot S(\Deq)$. The submatrix formed by these columns has the following block form (here $\omega'$ is our chosen primitive $2N$-th root on unity)

\beq
\left(\begin{array}{c}
T_0^{-1}\cdot S_0\\
T_0^{-1}\cdot\omega'S_0\\
\vdots\\
T_0^{-1}\cdot\omega'^{2N-1}S_0\\
0\\
\hline\\
T_1^{-1}\cdot S_1\\
\omega'^{-1}T_1^{-1}\cdot\omega'S_1\\
\vdots\\
\omega'^{1-2N}T_1^{-1}\cdot\omega'^{2N-1}S_1\\
\hline\\
\vdots\\\\
\hline\\
T_a^{-1}\cdot S_a\\
\omega'^{-a}T_a^{-1}\cdot\omega'S_a\\
\vdots\\
\omega'^{-a(2N-1)}T_a^{-1}\cdot\omega'^{2N-1}S_a\\
0\\
\hline
\vdots\\
\end{array}\right)=\left(\begin{array}{c}
T_0^{-1} S_0\\
\omega'T_0^{-1} S_0\\
\vdots\\
\omega'^{2N-1}T_0^{-1}S_0\\
0\\
\hline\\
T_1^{-1}S_1\\
T_1^{-1}S_1\\
\vdots\\
T_1^{-1}S_1\\
\hline\\
\vdots\\\\
\hline\\
T_a^{-1} S_a\\
\omega'^{1-a}T_a^{-1}S_a\\
\vdots\\
\omega'^{(1-a)(2N-1)}T_a^{-1}S_a\\
0\\
\hline
\vdots\\
\end{array}\right).
\eeq
Then we look at (\ref{e:sdrel}) and consider the $\O_\M\times \O_\M$ subblock of the equation corresponding to pairs of the simple objects $(M,\psi_M)$ for $M\in \O_\M$. The subblock on the left hand side of (\ref{e:sdrel}) is equal to $T_1S_1T_1$. The corresponding subblock on the right hand side of (\ref{e:sdrel}) is equal to 
\beq
\frac{\tau^+(\Deq)}{4N^2\dim\C}\cdot\sum\limits_{a\in\ZbN}{\overline{S_a}^TT_a^{-1}S_a+(\omega'^{-a})\overline{S_a}^TT_a^{-1}S_a+\cdots+(\omega'^{-a})^{2N-1}\overline{S_a}^TT_a^{-1}S_a}
\eeq
\beq
=\frac{\tau^+(\Deq)}{4N^2\dim\C}\cdot\sum\limits_{a\in\ZbN}{\left(1+\omega'^{-a}+\cdots+(\omega'^{-a})^{2N-1}\right)\overline{S_a}^TT_a^{-1}S_a}
\eeq
\beq
=\frac{\tau^+(\Deq)}{2N\dim\C}\overline{S_0}^TT_0^{-1}S_0.
\eeq
Now by \cite[Thm. 6.16]{DGNO}, $\tau^+(\Deq)=2N\cdot \tau^+(\C)$. Hence we conclude that
\beq
T_1S_1T_1=\frac{\tau^+(\C)}{\dim\C}\overline{S_0}^TT_0^{-1}S_0.
\eeq
Hence we have proved the equality \ref{e:tpt} and the proof is now complete.
\epf

\brk
The operator $\Theta:\KabC\rar{\cong} \KabC$ is an analogue of the Asai twisting operator (cf. \cite{S}, \cite[Prop. 3.4]{De2}). Then Theorem \ref{t:asai} is a categorical analogue of \cite[\S3.2]{De2}.
\erk

\end{document}